\documentclass[a4paper,11pt]{amsart}

\usepackage[utf8]{inputenc}  
\usepackage[T1]{fontenc}  
\usepackage{amscd,amssymb,amsmath,latexsym,graphicx,epsfig,color,url}
\usepackage{mathabx}
\usepackage[active]{srcltx}
\usepackage[all]{xy}
\usepackage{hyperref}
\hypersetup{breaklinks=true}

\usepackage{todonotes}

\newcommand{\ord}{\operatorname{ord}}

\newcommand{\C}{\mathbb{C}}
\newcommand{\F}{\mathbb{F}}

\newcommand{\N}{\mathbb{N}}

\newcommand{\Q}{\mathbb{Q}}
\newcommand{\R}{\mathbb{R}}

\newcommand{\Z}{\mathbb{Z}}

\newcounter{thm}
\numberwithin{thm}{section}
\numberwithin{equation}{section}

\theoremstyle{definition}
\newtheorem{definition}[thm]{Definition}

\newtheorem{remark}[thm]{Remark}

\newtheorem{algorithm}[thm]{Algorithm}

\theoremstyle{plain}

\newtheorem{proposition}[thm]{Proposition}
\newtheorem{theorem}[thm]{Theorem}
\newtheorem{corollary}[thm]{Corollary}

\newenvironment{ack}{\underline{\bf Acknowledgements:}}{}

\author[T. Cortadellas]{Teresa Cortadellas Ben\'itez}
\address{Departament de Matem\`atiques i Anal\'itica de Dades, IQS School of Engineering, Universitat Ramon Llull. Via Augusta 390, 08017 Barcelona, Spain}
\email{teresa.cortadellas@iqs.url.edu}
\title{On the effective Pourchet's Theorem}
\author[C. D'Andrea]{Carlos D'Andrea}
\address{Departament de Matem\`atiques i Inform\`atica, Universitat de Barcelona. Gran
Via 585, 08007 Barcelona, Spain \&  Centre de Recerca Matem\`atica, Edifici C, Campus Bellaterra, 08193 Bellaterra, Spain
}
\email{cdandrea@ub.edu}
\urladdr{http://www.ub.edu/arcades/cdandrea.html}

\author[A. B. de Felipe]{Ana Bel\'en de Felipe}
\address{Department of Mathematics, Barcelona East School of Engineering (EEBE), Universitat Polit\`ecnica de Catalunya - BarcelonaTech (UPC), Campus Diagonal Bes\`os, Ed. A, Avda. Eduard Maristany, 16, 08019 Barcelona, Spain.}
\email{ana.belen.de.felipe@upc.edu}

\author[J. Hurtado]{Joel Hurtado Moreno}
\address{Departament de Matem\`atiques, Universitat Polit\`ecnica de Catalunya, Av. Diagonal 647, 08028 Barcelona, Spain}
\email{joel.hurtado@estudiantat.upc.edu}
\author[M.E. Montoro]{M. Eul\`alia Montoro}
\address{Departament de Matem\`atiques i Inform\`atica, Universitat de Barcelona. Gran
Via 585, 08007 Barcelona, Spain}
\email{eula.montoro@ub.edu}
\begin{document}

\begin{abstract}
With the aid of Hensel Lemma, we refine the $2$-adic Newton polygon algorithm proposed in \cite{MKV23} to express computationally a given positive univariate polynomial with rational coefficients as a sum of five squares of rational polynomials -the effective Pourchet's Theorem-  and extend it to cover almost all the possible inputs.
We also provide examples which are covered with our methods but cannot be detected by previous conjectural algorithms.
\end{abstract}

%\subjclass[2010]{Primary 13P10; Secondary 13P15, 14M25}

%\keywords{Shape Lemma, resultant, subresultant, Poisson formula, toric variety}
\maketitle
\section{Introduction \& statement of the main results} \label{1}

Pourchet's Theorem \cite[Corollaire $4$]{pou71}  states that any polynomial $f(x)\in\Q[x]$ such that $f(t)>0$ for all $t\in\R$ can be expressed as
\begin{equation}\label{pct}
f(x)=f_1(x)^2+f_2(x)^2+f_3(x)^2+f_4(x)^2+f_5(x)^2,
\end{equation}
with $f_i(x)\in\Q[x],\,1\leq i\leq 5.$ This result, combined with the fact that there are polynomials with rational coefficients which cannot be sums of four squares (cf. \cite[Proposition $10$]{pou71}), shows that $5$ is the minimal number of squares of rational polynomials needed to express \emph{any} positive $f(x)\in\Q[x].$

Pourchet's original approach is based on local-global $p$-adic methods, in particular the Haase-Minkowski theorem, see \cite[Chapter $17$]{raj93} for a presentation of this result.  Hence, in principle it is not suitable for an algorithmic adaptation. 
Recently, Magron, Koprowski and Vaccon had developed algorithms to express rational polynomials as sums of two squares (\cite[Algorithm $1$]{MKV23}) and three or four squares (\cite[Algorithm $5$]{MKV23}) if that decomposition was possible.  They went further to propose that one can substract to a positive $f(x)\in\Q[x],$ a suitable even power (i.e. a rational square) of $\frac12$ to produce a decomposition in a sum of $4$ squares which could be treated algorithmically. This is the content of their Algorithm $6$  in \cite{MKV23}, which we describe in Algorithm \ref{algc} below. Their approach is proven to work provided that the $2$-adic valuation of either the constant or the leading coefficient of $f(x)$ is odd, and it is because of this obstruction that they are able to produce an algorithm (\cite[Algorithm $7$]{MKV23}) which decomposes any positive rational $f(x)$ as a sum of $6$ squares of rational polynomials, one more than Pourchet's sharp bound.

At the end of that paper, the authors  propose a procedure (\cite[Algorithm $9$]{MKV23} described in Algorithm \ref{alg9} in this text), which they conjecture succeeds in producing a decomposition of any positive rational polynomial as in \eqref{pct}.

The first main result of this paper is a negative response to this conjecture.
\begin{theorem}\label{ct1}
Algorithm \ref{alg9} does not stop after any finite number of steps for the following family of polynomials:
$$f_{k,N}(x):= \frac{4}{N^2}x^{2(2k+1)}+\frac{1}{N^2}x^{2k+1}+ \frac{4}{N^2}, $$
with $k=0, 1, \ldots, \ N\in\N$ odd, $N>64.$
\end{theorem}

The key aspect of this family of counterexamples is given in part by the fact that having  leading or constant coefficients of odd $2$-adic valuation  is a necessary condition for the algorithms in \cite{MKV23} to work, as the following extended family of examples show.

\begin{proposition}\label{dos}
Let $a\in\Z_{>0}.$
If $f(x)=g(x)^2+ 8a-1$ with $g(x)\in\Z[x]$ of degree $d=2k+1$ and  $k\in\Z_{\geq0},$  then the procedure of Algorithm \ref{algc} fails, i.e. the hypothesis on the $2$-adic valuation of the leading coefficient of $f(x)$ is necessary.
\end{proposition}

Even though Theorem \ref{ct1}  implies that\cite[Algorithm $9$]{MKV23} does not solve the effective Pourchet problem in its full generality, still we can show that a modification of it  actually works in more cases than those shown in that paper. For instance, if the degree of $f(x)$ is a multiple of $4,$ our modified Algorithm \ref{algn} works independently of the valuation of the leading coefficient of $f(x).$

\begin{theorem}\label{uno}
If the degree of $f(x)$ is a multiple of $4,$ then Algorithm \ref{algn}  stops after a finite number of steps and gives the right answer.
\end{theorem}

It  still remains  to find an algorithm to make explicit Pourchet's theorem for more cases than those covered so far. By using the Newton polygon method -as in that paper- we can cover more situations like the following.

\begin{theorem}\label{nos}
Let $f(x)\in\Z[x]$ be a polynomial of (even) degree $d$ such that $f(t)>0$ for all $t\in\R,$ and $f(0)=2^{2a}(4k+3)$ with $a, k\in\Z_{\geq0}$. Then, there exist $N,\,\ell\in\N,\, N$ odd such that 
$f(x)-\left(\frac{1}{2^\ell}x^{d/2}+\frac{2^a}{N}\right)^2$ is a sum of four squares.
\end{theorem}

Note that even though the family $f_{k,N}(x)$ of Theorem \ref{uno} does not fit directly with this statement, it turns out that the polynomial $N^2\cdot f_{k,N}(x-1)$ satisfies the hypothesis of Theorem \ref{nos}, and hence one cand find five elements in $\Q[x]$ such that the sum of their squares is equal to this polynomial. From here it is straightforward to get an expression of $f_{k,N}(x)$ as in \eqref{pct}.

One of the novelties of Theorem \ref{nos} is that we replace the constant $(\frac{1}{2^{l}})^2$ by the square of a polynomial of  degree $d/2$. We will see in the proof of this result that we also use the $2$-adic ``Newton polygon test''  which was the main tool in \cite{MKV23} to prove that our proposed input is a sum of four squares of rational polynomials.

It turns out that one can also use  the $p$-adic Hensel Lemma to get more flexibility and produce new decompositions of a positive $f(x)\in\Q[x]$ as a sum of five squares. The following theorems are proven by using this method. 
\begin{theorem}\label{gr4}
Let $f(x)\in\Z[x]$ be a square-free polynomial of degree $d=4k,\, k\in\N.$ If $f(t)>0$ for all $t\in\R,$   then there exists $\ell_0\in\N$ such that if $\ell\geq\ell_0,$ $$f(x)-\frac{1}{2^{2\ell}}(x^2+x+1)^{2k}$$ is a sum of four squares in $\Q[x].$
\end{theorem}

It is easy to verify that if a square-free $f(x)$ is such that  $f(t)>0$ for all $t\in\R,$  then  for all $g(x)\in\R[x]$ of degree bounded by $d,$  there exists $\ell_0\in\N$ such that if $\ell\geq\ell_0,\ f(t)-\frac{1}{2^{2\ell}}g(t)>0,$ and one can find this bound algorithmically, see Proposition \ref{kont} and Remark \ref{susi}.

\begin{theorem}\label{picky}
Let $f(x)\in\Z[x]$ be a square-free polynomial of degree $d=2(2k+1), \ k\in\N,$ such that $f(t)>0$ for all $t\in\R.$ Then, there exists $\ell_0\in\N$ such that if $\ell\geq\ell_0,$ 

$$f(x)-\frac{1}{2^{2\ell}}(x^2+x+1)^{2k}x^2$$
 is a sum of four squares if $f(0)$ is not of the form $2^{2a}(8b+1),$ with $a,b\in\Z_{\geq0}.$  If this is not the case, then this expression fails to be a sum of four squares for $\ell\gg0$ and $f(x).$
\end{theorem}

Integers of the form $2^{2a}(8b+1)$ are squares in $\Z_2,$ the ring of $2$-adic integers, and all squares in this ring are of this form (it follows from Proposition \ref{sqr}). The condition on the constant coefficient of $f(x)$ in the hypothesis of Theorem \ref{picky}  looks rather artificial, and one may argue that after a linear change of variables $f(x+\alpha),$ which does not change the structure of being a sum of four or five squares as we have done above with $f_{kN}(x)$ above,  we just have to look for $\alpha\in\Q$ such that $f(\alpha)$ is not a square in $\Q_2.$ If one finds such a change, then Theorem \ref{picky} would work in all the cases.
But unfortunately, this cannot always be done as the following  example due to Przemyslaw Koprowski shows it: 
\begin{equation}\label{exx} 4x^6+4x^3+9=(1+2x^3)^2+8
\end{equation} is a polynomial such that evaluated in any rational number gives a square in $\Q_2$ and it is not the square of any polynomial.  So, we cannot claim that with Theorems \ref{gr4} and \ref{picky} we cover all the positive rational polynomials, and extra work is needed in order to have a complete effective approach to Pourchet's Theorem.

The rest of the paper is organized as follows: in Section \ref{back} we establish the necessary notation and previous results needed for the proofs, which are then worked out in Section \ref{2}.
\smallskip

\begin{ack}
These results have been worked out at the Barcelona's Computational Algebra Seminar, hosted by the Math Institute of the University of Barcelona (IMUB). We are grateful to Przemyslaw Koprowski for useful conversations, and for sharing with us the example   \ref{exx}. T. Cortadellas is supported by MICINN research project PID2022-137283NB-C22. 
C. D'Andrea was supported by the Spanish State Research Agency, through the Severo Ochoa and Mar\'ia de Maeztu Program for Centers and Units of Excellence in R\&D (CEX2020-001084-M), and the European  H2020-MSCA-ITN-2019 research project GRAPES and the Spanish MICINN research project PID2023-147642NB-I00. A.B. de Felipe was partially supported by the AGAUR project 2021 SGR 00603. M.E. Montoro is partially supported by grant PID2021-124827NB-I00 funded by MCIN/AEI/ 10.13039/501100011033
and by ``ERDF A way of
making Europe'' by the ``European Union''.
\end{ack}

\bigskip
\section{Background \& Algorithms}\label{back}
In this section we introduce some necessary notation, and known results and algorithms which will be used for the proof of our main results. We start with univariate polynomials, and algorithmic criteria for locating their roots and simplicity.

Let $f(x)=c_0+c_1x+\ldots+c_{d-1}x^{d-1}+c_dx^d=c_d\prod_{j=1}^d (x-\lambda_j)$  with $c_d\neq0,$ be a polynomial with coefficients in $\C.$ For $k=0, 1, \ldots, 2d-2$, denote  
$$s_k=\lambda_1^k+\ldots +\lambda_d^k\in\C,
$$
and note that $s_k\in\Q(c_0,\ldots, c_d),$ to be more precisely each $s_k$ is actually a polynomial in $c_0,\ldots, c_{d-1}$ and has as denominator a power of $c_d$ (cf. (2.3) in \cite{US92}).

Set now
$$S_f=\begin{pmatrix}
s_0& s_1 &\ldots & s_{d-1}\\
s_1& s_2 & \ldots & s_d\\
\vdots &\vdots & \ldots & \vdots\\
s_{d-1} & s_d & \ldots & s_{2d-2}
\end{pmatrix}.
$$

\smallskip

Recall that by Sylvester's Law of Inertia (\cite{syl52}), any symmetric matrix with coefficients in a field $K$ of characteristic different from $2$ induces a quadratic form which can be diagonalized over $K.$ If $K\subset\R,$ the number of zeroes, positive and negative numbers in the diagonal is independent of the diagonalization. The rank of this matrix is the number of non-zero elements in the diagonal, and its signature the difference between the number of positive and negative numbers.
\begin{theorem}[\cite{jac35, jac36}]\label{hankel}
The number of distinct roots of $f(x)$ is equal to the rank of $S_f$. If in addition $f(x)\in\R[x],$ its number of distinct real roots is equal to $\sigma(S_f),$ the signature of the real symmetric matrix $S_f$.
\end{theorem}
For a modern proof of this result, see Theorem $2.2$ in \cite{US92}. As an easy consequence of this, we obtain the well-known fact that  all the roots of $f(x)\in\R[x]$ are real if and only if $S_f$ is definite positive. Also, none of the roots of $f(x)\in\R[x]$ is real if and only if the number of positive values in its diagonal form is equal to the number of negative values. 

Note that if we also have $f(x)\in\Q[x]$, then $S_f\in\Q^{d\times d}$ and effective criteria to compute both the rank and the signature of $S_f$ can be obtained by the standard computation of the sign sequence of the principal minors of this matrix, see for instance \cite{fra27}.

\begin{proposition}\label{kont}
Suppose that $f(x)\in\R[x]$ has degree $d$, is square-free and satisfies $f(t)>0$ for all $t\in\R.$ For any $g(x)\in\R[x]$ of degree bounded by $d,$ there exists $\varepsilon_0>0$ such that, if $\varepsilon_0\geq\varepsilon>0,$   $f(t)+\varepsilon\,g(t)>0 \ \forall t\in\R.$ 
\end{proposition}
\begin{proof}
For a new variable $s,$ the symmetric matrix $S_{f+s\, g}$ has its entries in $\R(s),$ and each of its coefficients is actually regular in $s=0$ (i.e. they can be evaluated in $s=0$). So, we have that $S_{f+s\,g}\to S_f$ when $s\to0.$ Here the matrix convergence is meant ``coefficient by coefficient''. Since $f(x)$ is square-free, the rank of $S_f$ is maximal by Theorem \ref{hankel}. Then, for $s$ small enough, the same will happen with $S_{f+s\,g}$. In addition, as there are no real zeroes for $f(x)$ and all the complex roots are different, $S_f$ has maximal rank, and there is an invertible real matrix $C$ such that  $C\cdot S_f\cdot C^T$ has $d/2$ positive values in the diagonal, the other $d/2$ values being negative as they must satisfy  $\sigma(S_f)=0$ thanks to Theorem \ref{hankel}. By continuity, we get that  each of the principal minors of $C\cdot S_{f+s\,g}\cdot C^T$ will have the same sign of the corresponding  principal minor of $C\cdot S_f\cdot C^T$ if $s$ is small enough.  This implies that $f(x)+\varepsilon_0\,g(x)$ has no real roots if $\varepsilon_0>0$ is small enough. In addition, as $f(0)>0,$ for small $\varepsilon_0$ we also have that $f(0)+\varepsilon\,g(0)>0$, which shows that this polynomial is positive everywhere. The result then follows by picking $1\gg\varepsilon_0>0$ satisfying all these conditions. 
\end{proof}
\begin{remark}\label{susi}
If both $f(x)$ and $g(x)$ are polynomials with rational coefficients, explicit bounds for $\varepsilon_0$ can be found in terms of $d$ and the absolute value of the coefficients of the input polynomials. Indeed, the principal $k$-th minor of $S_{f+s\,g}$ is a polynomial in $s$ of degree less than or equal to $kd$ having non-zero constant coefficient. One can then  get straightforwardly a bound for the value of small $s$ keeping this sign. 
\end{remark}

To test if a  given polynomial $f(x)$ with coefficients in a field is square-free, one uses the so-called {\em discriminant} of this polynomial, which can be defined as the resultant between $f(x)$ and $f'(x).$  We will denote it with $\mbox{disc}_x(f(x)).$ To compute it, one can use the determinant of a $(2d-1)\times (2d-1)$ matrix whose inputs are coefficients of $f(x)$ or of $f'(x)$, see Section $5$ of \cite {CLO07} for more on this. From this matrix expression one deduces straightforwardly the following result.
\begin{proposition}\label{dis}
If $f(x)\in\C[x]$ is square-free, and $g(x)\in\C[x]$ of degree bounded by the degree of $f(x),$ then for a new parameter $\lambda,\ \mbox{disc}_x(\lambda\, f(x)+g(x))$ is a nonzero polynomial in $\lambda$, with leading coefficient equal to $\mbox{disc}_x(f(x)).$ 
\end{proposition}

\bigskip
We now turn our attention to  local fields, and algorithms to detect sums of squares of polynomials with rational coefficients via testing them in these fields.
For a prime $p\in\N,$ we denote with $\Q_p$ the field of $p$-adic numbers, which is the completion of $\Q$ with respect to the topology given by the metric 
$d(a,b):=p^{-\ord_p{(a-b})}.$  The ring $\Z_p:=\{a\in\Q_p\:/\:d(a,0)\leq1\}$ of $p$-adic integers is a principal ideal domain with the unique maximal ideal $\{a\in\Z_p\:/\:d(a,0)<1\}.$

Finding roots of polynomials in $\Z_p[x]$ is an easier task as in $\Z[x],$ this is one of the reasons why working in this ring is convenient in several cases. For instance, the following $p$-adic version of Newton's method is going to be useful in the text.
\begin{theorem}\label{newton}\cite[Theorem $3$ in Section $5.2$]{BS66}
Suppose that $g(x)\in\Z_p[x],\,\gamma\in\Z_p,$ and $\delta\in\N$ are such that
\begin{itemize}
\item $g(\gamma)\equiv 0\ \ \mbox{mod} \  p^{2\delta+1}\Z_p,$
\item  $g'(\gamma)\equiv 0\ \ \mbox{mod} \  p^{\delta}\Z_p,$ and
\item  $g'(\gamma)\not\equiv 0\ \ \mbox{mod} \  p^{\delta+1}\Z_p,$ 
\end{itemize}
then there exists $\overline{\gamma}\in\Z_p$ such that $g(\overline{\gamma})=0,$ and $\overline{\gamma}\equiv\gamma\ \  \mbox{mod} \ p^{\delta+1}\Z_p.$
\end{theorem}
Irreducibility tests are also easier to find in the $p$-adics. For instance, the following statement is  ``Gauss'Lemma'' for polynomials in $\Z_p[x].$ 
\begin{proposition}\label{fact} \cite[Lemma $6.3.7$ and Problem $133$ ]{gou}
Suppose that $f(x)\in\Z_p[x]$  factors in a non-trivial way as a product $f(x) =g(x) \,h(x)$, with $g(x)$ and $h(x) \in \Q_p[x].$ Then, there exist non-constant polynomials   $g_0(x)$ and $h_0(x)\in\Z_p[x]$
such that $f(x)=f_0(x)\,g_0(x).$ If in addition $f(x),\,g(x)$ and $h(x)$ are monic, then both $g(x)$ and $h(x)$ are also in $\Z_p[x]$ (so one can take $g_0(x)$ as $g(x)$ and $h_0(x)$ as $h(x)$).  
\end{proposition}
\begin{corollary}\label{cori}
If a monic $f(x)\in\Z_p[x]$ has a root $x_0\in\Q_p$, then $x_0\in\Z_p.$ 
\end{corollary}

Another very useful tool for dealing with polynomials and factorizations in $\Z_p[x]$ is the so-called ``Hensel's Lemma for Polynomials''. To state it properly, recall that $\Z_p/p\Z_p\simeq \Z/ p\Z=:\F_p,$ the field of integers modulo $p.$ For $f(x)\in\Z_p[x],$ we denote with $[f(x)]\in(\Z_p/p\Z_p)[x]$ the polynomial  whose coefficients are the classes modulo $p$ of the coefficients of $f(x).$

\begin{theorem}\label{hensel}\cite[Theorem $4.7.2$ ]{gou}
Let $f(x)\in\Z_p[x]$
be a polynomial with coefficients in $\Z_p,$ and assume that there exist $g_1(x)$ and $h_1(x)\in\Z_p[x]$ such that
\begin{enumerate}
\item $g_1(x)$ is monic,
\item $[g_1(x)]$ and $[h_1(x)]$ are coprime, and
\item $[f(x)]=[g_1(x)][h_1(x)].$
\end{enumerate}

Then there exist polynomials $g(x), h(x) \in\Z_p[x]$ such that
\begin{enumerate}
\item $g(x)$ is monic,
\item  $[g(x)]=[g_1(x)],$ and $[h(x)]=[h_1(x)] ,$ and
\item $f(x) = g(x)h(x).$
\end{enumerate}
\end{theorem}

\begin{definition}
The {\em Newton diagram} $N(f)$ of $f(x)=\sum_{i=0}^n a_i x^i\in\Q_p[x]$ is the lower boundary of the convex hull of the subset of $\R^2$ consisting of all points $(i,\, \ord_p (a_i))$ for all $a_i\neq0.$
\end{definition}
\begin{definition}
We say that $f(x)\in\Q_p[x]$ is {\em pure} if $f(0)\neq0$ and $N(f)$ is a segment.
\end{definition}

There are very straightforward algorithms to deal with irreducible polynomials which are pure.

\begin{proposition}\label{GEC} \cite[``Generalized Eisenstein Criterion'' Lemma 3.5]{CG00}
Suppose that $f(x)\in\Q_p[x]$ is a pure polynomial with $N(f)$ having slope $k/\deg(f),$ where $\gcd(k,\,\deg(f))=1.$ Then, $f(x)$ is irreducible.
\end{proposition}

Another useful tool we will need is a characterization of squares in $\Q_2.$
\begin{proposition}\label{sqr} \cite[Problem 124]{gou}
A unit $u\in\Z_2$ is a square in $\Q_2$ if and only if $u\equiv 1\,\mbox{mod}\  8\Z_2.$
\end{proposition}

The connection between irreducibility and sums of squares of polynomials is given by the following result.

\begin{theorem}\label{padic}  \cite[Proposition $10$]{pou71}
Let $f(x)=\sum_{i=0}^d c_i x^i\in\Q[x]$ with $c_d\neq0.$ The following conditions are equivalent:
\begin{enumerate}
\item $f(x)$ is a sum of $4$ squares in $\Q[x];$
\item $c_d>0$ and for all prime factor $p(x)$ of $f(x)$ of odd multiplicity, $-1$ is a sum of two squares in  $\Q[x]/(p(x));$
\item $f(t)\geq0 \ \forall t\in\R,$ and in $\Q_2[x],$ every prime factor of $f(x)$ of odd multiplicity has even degree.
\end{enumerate}
\end{theorem}

We now turn to present the effective versions of Pourchet's Theorem given in \cite{MKV23}. We start by recalling their key algorithm on the effective decomposition of a given positive polynomial as a sum of squares when this is possible.

\begin{algorithm}\label{algd} \cite[Algorithm $5$ ``Decomposition  into a sum of $3$ of $4$ squares'']{MKV23}
\par\noindent\underline{Input:} A  polynomial $f(x)\in\Q[x]$, which is a priori known to be a sum of $3$ or $4$ squares.
\par\noindent\underline{Output:} Polynomials $f_1(x),\ldots, f_4(x)\in\Q[x]$ such that $f(x)=f_1(x)^2+\ldots +f_4(x)^2.$
\end{algorithm}

We now describe with more detail  two more of their algorithms, as we will need to use their steps in the text. From now on, $\log$ will denote logarithm in base $2.$

\begin{algorithm}\label{algc}  \cite[Algorithm $6$ ``Reduction to a sum of $4$ squares: odd valuation case'']{MKV23}
\par\noindent\underline{Input:} A positive square-free polynomial $f(x)=c_0+c_1x+\ldots +c_dx^d\in\Q[x].$ The $2$-adic valuations of the coefficients of $f(x)$ are $k_j:=\ord_2(c_j)$ for $0\leq j\leq d.$ Ensure $k_d$ is odd. It is assumed that $f(x)$ is not a sum of $4$ squares.
\par\noindent\underline{Output:} A polynomial $h(x)\in\Q[x]$ such that $f(x)-h(x)^2$ is a sum of $4$ or fewer squares.
\end{algorithm}
\begin{enumerate}
\item Find a positive number $\varepsilon$ such that $\varepsilon<\inf\{f(x),\,x\in\R\}.$
\item Set $l_1:=\lceil-\frac12\cdot\log\varepsilon\rceil.$
\item Set $l_2:=\lceil-\frac{k_0}2\rceil+1.$
\item Set $l_3:=\left\lceil\max\left\{\frac{jk_d-dk_j}{2d-2j},\ 0<j<d\right\}\right\rceil.$
\item Initialize $l:=\max\{l_1,l_2,l_3\}.$
\item While $\gcd(d,2l+k_d)\neq1$ do $l:=l+1.$
\item Return $h(x):=2^{-l}.$
\end{enumerate}

\medskip

Both algorithms \ref{algd} and \ref{algc} have been shown to work in \cite{MKV23}. At the end of that paper, the authors propose the following `procedure, show that it works experimentally in all the cases, and conjecture their validity in general.

\begin{algorithm}\label{alg9} \cite[Algorithm $9$ ``Reduction to sum of $4$ squares'']{MKV23}
\par\noindent\underline{Input:} A positive square-free polynomial $f(x)=c_0+c_1x+\ldots +c_dx^d\in\Q[x].$
\par\noindent\underline{Output:} A polynomial $h(x)\in\Q[x]$ such that $f(x)-h(x)^2$ is a sum of $4$ or fewer squares.
\end{algorithm}

\begin{enumerate}
\item If $f(x)$ is a sum of $4$ squares, then return $h(x):=0.$
\item Set $f_*(x)=c_d+c_{d-1}x+\ldots+c_0x^d.$
\item Find a positive number $\varepsilon$ such that
$$\varepsilon<\inf\{f(x),\ x\in\R\} \ \mbox{and} \ \varepsilon<\inf\{f_*(x), \ x\in\R\}.
$$
\item Initialize $l:=\lceil-\frac12\cdot\log\varepsilon\rceil.$
\item While True do
\begin{enumerate}
\item if $f(x)-2^{-2l}$ is irreducible in $\Q_2[x]$, then return $h(x):=2^{-l}$.
\item if $f(x)-2^{-2l}x^d$ is irreducible in $\Q_2[x],$ then return $h(x):=2^{-l}x^{d/2}.$
\item $l:=l+1.$
\end{enumerate}
\end{enumerate}

\medskip

We conclude this section by proposing our extension of Algorithm \ref{algc} to cover all polynomials $f(x)\in\Q[x]$ of degree multiple of $4.$ Note that the main difference is given in the computation of the $\gcd$ in one of the lines of the algorithm.

\begin{algorithm}\label{algn}  $^{}$
\par\noindent\underline{Input:} A positive square-free polynomial $f(x)=c_0+c_1x+\ldots +c_dx^d\in\Q[x]$ of degree $d=4d_0,$ which is not a sum of $4$ squares.
The $2$-adic valuations of the coefficients of $f(x)$ are $k_j:=\ord_2(c_j)$ for $0\leq j\leq d.$
\par\noindent\underline{Output:} A polynomial $h(x)\in\Q[x]$ such that $f(x)-h(x)^2$ is a sum of $4$ or fewer squares.
\end{algorithm}
\begin{enumerate}
\item If $k_d$ is odd, then apply Algorithm \ref{algc} to the input. Otherwise,
\begin{enumerate}
\item Find a positive number $\varepsilon$ such that $\varepsilon<\inf\{f(x),\,x\in\R\}.$
\item Set $l_1:=\lceil-\frac12\cdot\log\varepsilon\rceil.$
\item Set $l_2:=\lceil-\frac{k_0}2\rceil+1.$
\item Set $l_3:=\left\lceil\max\left\{\frac{jk_d-dk_j}{2d-2j},\ 0<j<d\right\}\right\rceil.$
\item Initialize $l:=\max\{l_1,l_2,l_3\}.$
\item While $\gcd(d,2l+k_d)\neq2$ do $l:=l+1.$
\item Return $h(x):=2^{-l}.$
\end{enumerate}
\end{enumerate}

\section{The proofs}\label{2}
In this section we provide the proofs of the results announced in the introduction.

\begin{proof}[Proof of Theorem \ref{ct1}]
Note that
$f_{k,N}(x)=\Big(\frac{2}{N} x^{2k+1}+\frac{1}{4N}\Big)^2+\frac{63}{2^4\cdot N^2}.
$  

From Proposition \ref{sqr}, we have that $\frac{63}{2^4\cdot N^2}=-\alpha^2,$ with $\alpha\in\Q_2,$  so   from the factorization $$f_{k,N}(x)=\Big(\frac2N x^{2k+1}+\frac{1}{4N}+\alpha\Big)\Big(\frac2N x^{2k+1}+\frac{1}{4N}-\alpha\Big),$$ 
we deduce that this polynomial is square-free and must have an irreducible factor of odd degree. Hence, $f_{k,N}(x)$ is not a sum of $4$ squares thanks to Theorem \ref{padic}. So, Algorithm \ref{alg9} does not stop at step (1).

To continue, notice that  $(f_{k,N})_* =f_{k,N},$ so $\varepsilon$  in (3) of Algorithm \ref{alg9} can be taken as any number $0<\varepsilon<\frac{63}{2^4\cdot N^2}.$ Therefore we have that $l$ should be initialized as $6$ or a larger value if $N>64.$ For all these cases, we have
\begin{equation}\label{fkn}
f_{k,N}(x)-\frac{1}{2^{2l}}=\Big(\frac2N x^{2k+1}+\frac{1}{4N}\Big)^2+\frac{63}{2^4\cdot N^2}-\frac{1}{2^{2l}}= \Big(\frac2N x^{2k+1}+\frac{1}{4N}\Big)^2-\alpha_l^2,
\end{equation}
where $\alpha_l\in\Q_2$ because $\frac{63}{2^4\cdot N^2}-\frac{1}{2^{2l}}=\frac{2^{2(l-2)}\cdot 63-N^2}{N^2\cdot2^{2l}},$  and $2^{2(l-2)}\cdot 63-N^2$ is congruent to $-1$ modulo $8\Z_2$ (recall that $l\geq6$).  Here, we are using Proposition \ref{sqr} again. From Theorem \ref{padic}, we deduce then that the algorithm will not stop at (5) a because \eqref{fkn} factors as a product of two polynomials of odd degree. The case (5) b is similar, as  $(f_{k,N})_* =f_{k,N}.$ This concludes with the proof of the claim.
\end{proof}

\bigskip
\begin{proof}[Proof of Proposition \ref{dos}]
Note first that $f(t)>0$ for all $t\in\R$, and that $k_{2d}$ is even.  We will check that this is the only restriction of the input that $f(x)$ does not fulfill. 

To do so, note that  $-8a+1$ is a square in $\Q_2$ thanks to Proposition \ref{sqr}. So, there exists $\alpha\in\Q_2$ such that
$$f(x)=(g(x)+\alpha)\cdot(g(x)-\alpha).
$$
Each of the factors $(g(x)\pm\alpha)$ has degree $2k+1$, so they are either irreducible or have an irreducible factor of odd degree.  From Theorem \ref{padic}, we then deduce that $f(x)$ is not a sum of $4$ squares in $\Q[x]$, so the input of the algorithm is correct except for $k_{2d}$ being even.

To continue, note that if $l\geq2,$
$$
\begin{array}{ccl}
2^{2l}\big(f(x)-\frac{1}{2^{2l}}\big)& = & 2^{2l}f(x)-1=2^{2l}g(x)^2+(2^{2l}(8a-1)-1)
\\ & =&  2^{2l}g(x)^2-\alpha_l^2=(2^lg(x)+\alpha_l)(2^lg(x)-\alpha_l),
\end{array}$$
with $\alpha_l\in\Q_2$ thanks to Proposition \ref{sqr}. We then have that this expression has an irreducible factor of odd multiplicity and odd degree, so again thanks to Theorem \ref{padic} we deduce that $f(x)-\frac{1}{2^{2l}}$ is not a sum of four squares in $\Q[x]$ for all $l\geq2.$
\end{proof}

\bigskip
\begin{proof}[Proof of Theorem \ref{uno}]
If $k_d$ is odd, then the algorithm clearly holds. Suppose now that  $k_d=2k.$  
Then we have that $$\gcd(d,2l+k_d)=\gcd(4d_0,2l+2k)=2\gcd(2d_0, l+k)$$ is always a multiple of $2$ and at some point -at most up to $2d_0$ iterations, it will be $2.$ Denote $l_0$ the value of $l$ that satisfies this. Then, the Newton diagram of $f(x)-\frac{1}{2^{2l_0}}$ is a segment with three lattice points: $(0,-2l_0), (2d_0,k-l_0), (4d_0, 2k).$

From here we deduce that, either $f(x)-\frac{1}{2^{2l_0}}$ is irreducible in $\Q_2[x]$ (and hence this polynomial is the sum of $4$ squares in $\Q[x]$ thanks to Theorem \ref{padic}), or $f(x)-\frac{1}{2^{2l_0}}$ factors in $\Q_2[x]$  as the product of two polynomials whose Newton diagrams are parallel to the segment with endpoints $(0,0), (2d_0,k+l_0),$ and have the same length as it. As $\gcd(2d_0, k+l_0)=1,$ there are no interior lattice points in this segment, so each of these factors are irreducible thanks to Proposition \ref{GEC}. So,  $f(x)-\frac{1}{2^{2l_0}}$ actually factors in $\Q_2[x]$ as the product of two irreducible factors of even degree $2d_0$. By using Theorem \ref{padic} again, we conclude that it is a sum of four squares in $\Q[x]$. 
\end{proof}

\bigskip
\begin{proof}[Proof of Theorem \ref{nos}]
As in the proof of Proposition \ref{kont}, it can be shown $g(x):=f(x)-\left(\frac{1}{2^\ell}x^{d/2}+\frac{2^a}{N}\right)^2$ is positive everywhere for $N,\,\ell$ big enough, as the coefficients of $\left(\frac{1}{2^\ell}x^{d/2}+\frac{2^a}{N}\right)^2$ tend to $0$ when $\ell,\,N\to\infty.$

We choose a large odd $N\in\N,$  and  compute  the $2$-adic valuation of the constant coefficient of $g(x)$: as $N^2\equiv 1 \, \mbox{mod} \,4,$ it is equal to
$$f(0)-\frac{2^{2a}}{N^2}=\frac{2^{2a}(4k'+2)}{N^2}=\frac{2^{2a+1}(2k'+1)}{N^2},
$$
for a suitable $k'\in\N.$ So its valuation is equal to $2a+1.$ Working around the leading coefficient of $g(x),$  we easily get that its  $2$-adic valuation is equal to  $-2\ell$. Note also that there is another coefficient of $g(x)$ with very low valuation if $\ell\gg 0:$ the one corresponding to $x^{d/2}$, with valuation $a+1-\ell.$

To compute the Newton diagram of $g(x),$ one should pay attention to the relative position of the points 
$(0, 2a+1),\  (d/2, a+1-\ell),$ and $(d, -2\ell).$ The point in the middle lies above the segment $S$ with endpoints $(0, 2a+1)$ and $(d, -2\ell).$ In addition, the points of the form $(i,\ord_2(c_i)),$ where $c_i$ is a non-zero coefficient of $f(x),$ belong to the first quadrant. We conclude that the segment $S$ is equal to $N(g)$ for $\ell$ big enough. One can then choose $\ell$ big enough such that $S$ has no interior lattice points, which then proves that $g(x)$ is irreducible in $\Q_2[x]$ (see Proposition \ref{GEC}) and hence the result follows from Theorem \ref{padic}.
\end{proof}

\bigskip
\begin{proof}[Proof of Theorem \ref{gr4}]
From Proposition \ref{kont} applied to $f(x)=\sum_{i=0}^{4k} c_i x^i$ and $g(x)=-(x^2+x+1)^{2k}$, we deduce that there exists $\ell_0\in\N$ such that if $\ell\geq\ell_0$ then $f(t)-\frac{1}{2^{2\ell}}(t^2+t+1)^{2k}>0$ for all $t\in\R$. Next, we factor $2^{2\ell}f(x)-(x^2+x+1)^{2k}$ as a product of irreducible polynomials in $\Q_2[x]$: as its leading coefficient is  $u:=2^{2\ell}c_{4k}-1,$ which is a unit in $\Z_2,$ we then have that 
\begin{equation}\label{tic}
2^{2\ell}f(x)-(x^2+x+1)^{2k}=u\,\prod_{i=1}^N p_i(x)^{r_i},
\end{equation}
with  $p_i(x)\in\Q_2[x]$ monic and irreducible, $ 1\leq i\leq N.$ This implies, thanks to Proposition \ref{fact}, that the factorization \eqref{tic} takes place in $\Z_2[x].$ We then take classes in both sides in the quotient ring $\Z_2[x]/2\Z_2[x]\simeq \F_2[x],$ to obtain
$$[x^2+x+1]^{2k}=\prod_{i=1}^N [p_i(x)]^{r_i}.
$$
The polynomial $[x^2+x+1]$ is irreducible in $\F_2[x],$ and note that each of the $[p_i(x)]$ has the same degree as $p_i(x)$ (because the latter is monic), which implies that each of them must be a power of a quadratic polynomial in the quotient ring. We deduce then that all the $p_i(x)$'s have even degree, and the claim now follows from Theorem  \ref{padic}.
\end{proof}

\bigskip
\begin{proof}[Proof of Theorem \ref{picky}]
Write $f(x)=\sum_{j=0}^d c_j x^j$ and denote $k_0=\ord_2(c_0)$ and $k_1=\ord_2(c_1)$. By Proposition \ref{kont}, there exists $\ell'\in\N$ such that if $\ell\geq\ell'$ then $f(t)-\frac{1}{2^{2\ell}}(t^2+t+1)^{2k}t^2>0$ for all $t\in\R$. 

Suppose first that $k>0.$ Choose $\ell\in\N$ such that $\ell>\max\{1,\ell',\,k_0/2-k_1+2\}$. Denote $q(x):=2^{2\ell}f(x)-(x^2+x+1)^{2k}x^2$. By using Hensel's Lemma (Theorem \ref{hensel}) with $(x^2+x+1)^{2k}$ and $x^2$, we deduce that $q(x)=g(x)h(x),$ where $g(x)\in\Z_2[x]$ has degree $4k$, $h(x)\in\Z_2[x]$ has degree $2,$ and $[g(x)]= [x^2+x+1]^{2k}$ and $[h(x)]= [x]^2$ in $\F_2[x].$
Reasoning as in the proof of Theorem \ref{gr4}, we deduce that all the irreducible factors of $g(x)$ have even degree. The proof of the claim would be complete -Thanks to Theorem \ref{padic} again- if we show that $h(x)$ is irreducible in $\Q_2[x]$. For this it will be enough to show that $q(x)$ does not have roots in $\Q_2$.

Suppose that this is not the case, and denote with $\alpha_\ell\in\Q_2$ a root of $q(x)$. The leading coefficient of $q(x)$ equals $u:=2^{2\ell}c_d-1$, which is a unit in $\Z_2$. Since $\alpha_\ell$ is a root of the monic polynomial $u^{-1}q(x)\in\Z_2[x]$, by Corollary \ref{cori} we get that $\alpha_\ell\in\Z_2$. Hence we can write $\alpha_\ell=2^{a_\ell}(1+2q_\ell)$, with $a_\ell\in\Z_{\geq0}$ and $q_\ell\in\Z_2.$ 

We then have
\begin{equation}\label{cuzi}
2^{2\ell}f(\alpha_\ell)=(\alpha_\ell^2+\alpha_\ell+1)^{2k}\alpha_\ell^2=2^{2a_\ell}(1+2q_\ell)^2(\alpha_\ell^2+\alpha_\ell+1)^{2k}.
\end{equation}
Since $\alpha_\ell^2+\alpha_\ell+1\neq0$, the $2$-adic valuation of the right-hand-side of \eqref{cuzi} is equal to $2a_\ell.$ On the left hand side, we have
\begin{equation}\label{cuz}
2^{2\ell}f(\alpha_\ell)=2^{2\ell+k_0}\tilde{c}_0+2^{2\ell+k_1+a_\ell}\tilde{c}_1(1+2q_\ell)+2^{2a_\ell}\sum_{j=2}^dc_j (1+2q_\ell)^j2^{2\ell+a_\ell(j-2)},
\end{equation}
where $\tilde c_0=2^{-k_0}c_0,\tilde c_1=2^{-k_1}c_1\in\Z$.
As $\ell>0$, the third of the three summands of the right-hand-side of \eqref{cuz} has $2$-adic valuation strictly larger than $2a_\ell$. So we must have either
$2\ell+k_1+a_\ell=2a_\ell\leq 2\ell+k_0$ or $2\ell+k_0=2a_\ell<2\ell+k_1+a_\ell.$
For the first case we have $a_\ell=2\ell+k_1.$ But note then that this implies that $2a_\ell=4\ell+2k_1\leq2\ell+k_0\iff 2\ell\leq k_0-2k_1,$ which contradicts the choice of $\ell$. 

So, we must have that $a_\ell=\ell+\frac{k_0}{2}<2\ell+k_1,$ which holds thanks to the hypothesis on $\ell$. Note that this implies that $k_0$ must be even. We now factor $2^{2a_\ell}$ from each side of \eqref{cuzi} and we obtain (recall that $2a_\ell=2\ell+k_0$):
\begin{equation}\label{eq:factor}
    \tilde{c}_0+2^{k_1+a_\ell-k_0}\tilde{c}_1(1+2q_\ell)+\sum_{j=2}^dc_j (1+2q_\ell)^j2^{2\ell+a_\ell(j-2)}=((1+2q_\ell)(\alpha_\ell^2+\alpha_\ell+1)^{k})^2.
\end{equation}
As $k_1+a_\ell-k_0=k_1+\ell-k_0/2>2$ because of our choice of $\ell>1,$  we deduce that the left-hand-side of \eqref{eq:factor} is congruent to $\tilde{c}_0\ \mbox{mod}\,8\Z_2.$ The right-hand-side of \eqref{eq:factor} is the square of a unit in $\Z_2$, thus it is congruent to $1$ modulo $8\Z_2.$
Therefore, $c_0=2^{k_0}(8b+1)$ with $k_0\in 2\Z$ and $b\in\Z_{\geq0}$. This concludes the proof of the claim if $k>0.$ 

For the case $k=0,$ take $\ell\in\N$ such that $\ell>\max\{2,\ell',(k_0+5)/2\}$. We only have to check that the quadratic polynomial $q(x):=2^{2\ell}f(x)-x^2$ does not have roots in $\Q_2$ if $c_0$ is as in the hypothesis. But it is easy to check that the discriminant $\Delta$ of $q(x)$ is equal to
\begin{equation}\label{disc}
\Delta=2^{4\ell}c_1^2-2^{2\ell+2}(2^{2\ell}c_2-1)c_0=2^{2\ell+2+k_0}(2^{2\ell-2-k_0}c_1^2-(2^{2\ell}c_2-1)\tilde c_0),
\end{equation}
where $\tilde c_0=2^{-k_0}c_0$ is a unit of $\Z_2$, and $\ord_2\Delta=2\ell+2+k_0$. If $q(x)$ has a root in $\Q_2$, then $\Delta$ is a square in $\Q_2$. We deduce then that $k_0$ must be even. Hence $c_0=2^{2a}\tilde c_0$ for some $a\in\Z_{\geq0}$. Moreover, $2^{2\ell-2-k_0}c_1^2-(2^{2\ell}c_2-1)\tilde c_0\in\Z_2$ is a unit that is a square in $\Q_2$ and is congruent to $\tilde c_0$ modulo $8\Z_2$. Proposition \ref{sqr} implies that $\tilde c_0$ is of the form $8b+1$ with $b\in\Z_{\geq0}$. The claim follows for this case.

%The expression between parenthesis is congruent to $c_0$ modulo $8$ because $\ell\geq3.$
%It is easy to verify then that if $c_0$ is not of the form $2^{2a}(8b+1),$ the whole expression cannot be a square in $\Q_2$ for large $\ell$ because if its $2$-adic valuation is odd, the same will happen with its discriminant \eqref{disc}. 

%If not, then from  $c_0=2^{2a}\tilde{c_0},$ with $\tilde{c_0}$ being a unit in $\Z_2,$ we deduce that,   \eqref{disc} equals 
%$2^{2\ell+2+2a}(2^{2\ell-2-2a}c_1^2-(2^{2\ell}c_2-1)\tilde{c_0}).$
%The expression between parenthesis is congruent to $\tilde{c_0}$ modulo $8\Z_2$ for large $\ell.$ From Proposition \ref{sqr} we deduce that \eqref{disc} is a square in $\Q_2$ if and only if $\tilde{c_0}\equiv 1 \, \mbox{mod}\, 8\Z_2.$   Hence the claim follows for this case.

To conclude, we need to prove that if $c_0=2^{2a}(8b+1)$ and $f(x)$ is square-free, then $g(x):=f(x)-\frac{1}{2^{2\ell}}(x^2+x+1)^{2k}x^2$ is not a sum of four squares for $\ell\gg0.$ Indeed, if we set $\gamma=2^{\ell+a},\,\delta=\ell+a+1$, and $q(x)=2^{2\ell}f(x)-(x^2+x+1)^{2k}x^2$, then we deduce that 
$$\begin{array}{ccl}
q(\gamma) &= &2^{2\ell}\sum_{j=0}^d c_ j\,2^{(\ell+a)j}-2^{2(\ell+a)}\big(2^{2(\ell+a)}+2^{\ell+a}+1 \big)^{2k}\\
 & =& 2^{2(\ell+a)}\big(\sum_{j=1}^d c_j 2^{(\ell+a)j-2a} +8b+1- (2^{2(\ell+a)}+2^{\ell+a}+1)^{2k}\big).
\end{array}$$ 
If $\ell\geq a+3,$ it is straightforward to verify that the expression between parenthesis above is a multiple of $8,$ which shows then that $q(\gamma)\equiv 0 \ \mbox{mod}\, 2^{2\delta+1}\Z_2.$ 
In addition, it is not hard to verify that $q'(\gamma)\equiv 0\ \ \mbox{mod} \ 2^{\delta}\Z_2,$ and  $q'(\gamma)\not\equiv 0\ \ \mbox{mod} \ 2^{\delta+1}\Z_2.$ From Theorem \ref{newton} we deduce then that there exists $\overline{\gamma}\in\Z_2$ which is a root of $g(x)$. If this is a simple root, then because of Theorem \ref{padic} this polynomial cannot be a sum of four squares, as it has a linear irreducible factor in its factorization.  If $\overline{\gamma}$ is a mutiple root, then $2^{2\ell}$ would be  a root of  $p(\lambda):=\mbox{disc}_x(\lambda\,f(x)-(x^2+x+1)^{2k}x^2)\in\Z[\lambda].$ As $f(x)$ is square-free, from Proposition \ref{dis} we deduce that $p(\lambda)$ is not identically zero, and hence it has a finite number of roots. If $\ell\gg0$ then it is clear that $p(2^{2\ell})\neq0,$ and hence the claim follows.
\end{proof}

\begin{remark}
The conditions on $c_0$ imposed in the hypothesis of Theorem \ref{picky} are necessary, as if we pick for instance $f(x)=x^2+1,$ for all $\ell>1,$ thanks to Proposition \ref{sqr} we have a factorization  $(2^{2\ell}-1)x^2+2^{2\ell}=(2^{\ell}-\alpha_\ell x)(2^{\ell}+\alpha_\ell x)$ with $\alpha_\ell\in\Z_2.$ From Theorem \ref{padic} we deduce then that this polynomial is not a sum of $4$ squares.
\end{remark}

\end{document}